\documentclass[reqno]{amsart}

\usepackage{amsmath, , amsfonts, amssymb}
\usepackage[english]{babel}
\usepackage[applemac]{inputenc}
\usepackage[T1]{fontenc}
\usepackage{graphicx,pdflscape}
\usepackage{layout}
\usepackage[all]{xy}
\usepackage{hyperref}
\usepackage[font=small,labelsep=none]{caption}

\newtheorem{lem}{Lemma}
\newtheorem{thm}{Theorem}
\newtheorem{cor}{Corollary}[thm]

\theoremstyle{definition}
\newtheorem{defn}{Definition}

\theoremstyle{remark}

\title[Number of strings on essential tangle decompositions]{The number of strings on essential tangle decompositions of a knot can be unbounded}
\author{Jo\~{a}o Miguel Nogueira}
\address{CMUC, Department of Mathematics\\
   University of Coimbra\\
   Apartado 3008 EC Santa Cruz\\
   3001-501 Coimbra\\
   Portugal}
\email{nogueira@mat.uc.pt}

\thanks{This work was partially supported by the Centro de Matem\'{a}tica da
Universidade de Coimbra (CMUC), funded by the European Regional
Development Fund through the program COMPETE and by the Portuguese
Government through the FCT - Funda\c{c}\~{a}o para a Ci\^{e}ncia e a Tecnologia
under the project PEst-C/MAT/UI0324/2011.}


\begin{document}

\maketitle

\begin{abstract}
We construct an infinite collection of knots with the property that any knot in this family has $n$-string essential tangle decompositions for arbitrarily high $n$.
\end{abstract}

\section{Introduction}
A \textit{$n$-string tangle} $(B, \mathcal{T})$ is a ball $B$ together with collection of $n$ disjoint arcs $\mathcal{T}$ properly embedded in $B$, for $n\in \mathbb{N}$. We say that $(B, \mathcal{T})$  is \textit{essential}, if $n$ is $1$ and its arc is knotted\footnote{An arc of $\mathcal{T}$ is \textit{unknotted} if it co-bounds a disk embedded in $B$ together with an arc in $\partial B$, otherwise it is said to be \textit{knotted}.}, or if $n$ is bigger than $1$ and there is no properly embedded disk in $B$ disjoint from $\mathcal{T}$ and separating the components of $\mathcal{T}$ in $B$. Otherwise, we say that the tangle is \textit{inessential}. (See Figure \ref{tangles.pdf} for examples.)\\
Let $K$ be a knot in $S^3$ and $S$ a $2$-sphere in general position with $K$. Each ball bounded by $S$ in $S^3$ intersects $K$ in the same number $n$ of arcs. So, these balls together with the arcs of intersection with $K$ are $n$-string tangles. In this case, we say that $S$ defines a \textit{$n$-string tangle decomposition} of $K$, and if both tangles are essential we say that the tangle decomposition of $K$ defined by $S$ is $\textit{essential}$. A knot is composite if, and only if, it has a $1$-string essential tangle decomposition, otherwise the knot is prime. Note also that $S$ defines an essential tangle decomposition for $K$ if, and only if, the intersection of $S$ with the exterior of $K$, $E(K)$\footnote{We denote the \textit{exterior} of a knot $K$, that is  $S^3-int\,N(K)$ where $N(K)$ is a regular neighborhood of $K$, by $E(K)$.}, is an essential surface in $E(K)$. (See Definition \ref{essential}.)\\

\begin{figure}[htbp]
\centering
\includegraphics{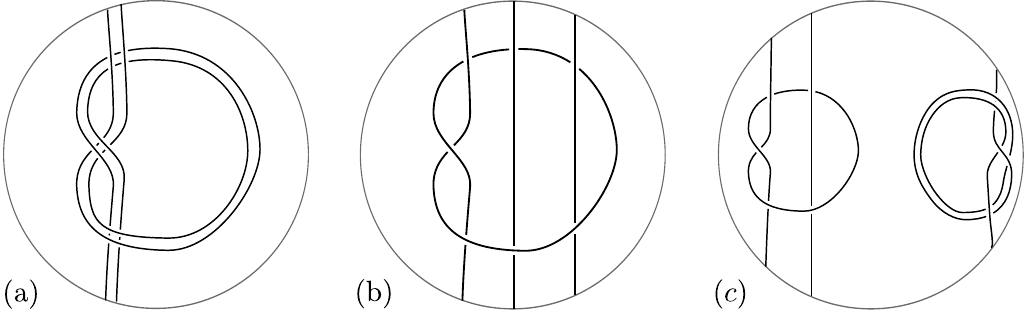}
\caption{: Examples of essential tangles, in (a) and (b), and an inessential tangle, in (c).}
\label{tangles.pdf}
\end{figure}

A tangle decomposition of a knot is natural and has been relevant for knot theory and its applications. The concept of ``tangle'' was first used in the work of Conway \cite{Conway}, where he defines and classifies ($2$-string) rational tangles and uses it as an instrument to list knots. The concept of essential tangle was first  used in \cite{Kir-Lick} where Kirby and Lickorish prove that any knot is concordant to a prime knot. They actually define \textit{prime tangle}, that is an essential tangle with no local knots\footnote{A tangle $(B, \mathcal{T})$ has no local knots if any $2$-sphere intersecting $\mathcal{T}$ transversely in two points bounds a ball in $B$ meeting $\mathcal{T}$ in an unknotted arc.}. Another example is the work of Lickorish in \cite{Lickorish} where he proves for instance that if a knot has a $2$-string prime tangle decomposition then the knot is prime. Tangles are also used in applied mathematics to study the DNA topology. The paper \cite{Buck} by Buck surveys the subject concisely, and also explains how tangles are useful to the study of the topological properties of DNA, an application pioneered by Ernst and Sumners in \cite{Ernst-Sumners}.\\

This paper addresses the question if the number of strings on essential tangle decompositions of a fixed knot is bounded. There are results showing some evidence for this to be true. For instance, knots with no closed essential surfaces \cite{CGLS}, tunnel number one knots \cite{Gordon-Reid} and free genus one knots \cite{Matsuda-Ozawa} have no essential tangle decompositions. There also are knots with an unique essential tangle decomposition \cite{Oz}. Furthermore, in Proposition 2.1 of \cite{Mizuma-Tsutsumi}, Mizuma and Tsutsumi proved that for a given knot the number of strings in essential tangle decompositions, without parallel strings\footnote{Two strings of a tangle in a ball $B$ are parallel if there is an embedded disk in $B$ co-bounded by these strings and two arcs in $\partial B$.}, is bounded. The proof of this result  allows a more general statement. That is, the number of strings that are not parallel to other strings in an essential tangle decomposition of a fixed knot is bounded. So, from this flow of results and intuition on essential tangle decompositions the following theorem and its corollary are surprising.

\begin{thm}\label{main}
There is an infinite collection of prime knots such that for all $n\geq 2$ each knot has a $n$-string essential tangle decomposition.
\end{thm}

\begin{cor}\label{cor}
There is an infinite collection of knots such that for all $n\geq 1$ each knot has a $n$-string essential tangle decomposition.
\end{cor}

Essential surfaces are very important in the study of $3$-manifold topology. And as observed above, to each $n$-string essential tangle decomposition of a knot corresponds a meridional essential surface in the exterior of the knot, with $2n$ boundary components.  Therefore, from the results in this paper there are knots with meridional planar essential surfaces in their exteriors with all possible numbers of boundary components. Furthermore, from Lemma 1.2 in \cite{Bleiler}, the double cover of $S^3$ along these knots contains genus $g$ closed incompressible surfaces, meeting the fixed point set of the covering action in $2(g+1)$ points, and separating the double cover in irreducible and $\partial$-irreducible components, for all $g\geq1$.\\

The reference used for standard definitions and results of knot theory is Rolfsen's book \cite{Rolfsen}, and throughout this paper we work in the piecewise linear category.\\
In Section \ref{handlebody}, we show the existence of handlebody-knots (see Definition \ref{handlebody-knot}) with incompressible planar surfaces in their exteriors with $b$ boundary components, for all $b\geq 2$. In Section \ref{examples}, we use these handlebody-knots to prove Theorem \ref{main} and its corollary. The main techniques used are standard in $3$-manifold topology. Along the paper, the number of connected components of a topological space $X$ is denoted by $|X|$.

\section{Meridional incompressible planar surfaces in handlebody-knots complements}\label{handlebody}

To prove Theorem \ref{main} we use the correspondence between $n$-string essential tangle decompositions of a knot and meridional planar essential surfaces in the knot exterior. So, we start by defining these surfaces.

\begin{defn}\label{essential}
A \textit{planar surface} is a surface obtained from a $2$-sphere by removing the interior of a finite number of disks.\\
Let $H$ be a handlebody embedded in $S^3$.\\
A surface $P$ properly embedded in $E(H)=S^3-int\, H$ is \textit{meridional} if each boundary component of $P$ bounds a disk in $H$.\\
An embedded disk $D$ in $E(H)$ is a \textit{compressing disk} for $P$ if $D\cap P= \partial D$ and $\partial D$ does not bound a disk in $P$. We say that $P$ is \textit{incompressible} if there is no compressing disk for $P$ in $E(H)$.\\
An embedded disk $D$ in $E(H)$ is a \textit{boundary compressing disk} for $P$ if $\partial D\cap P=\alpha$, with $\alpha$ a connected arc not cutting a disk from $P$, and $\partial D-\alpha=\beta$ a connected arc in $\partial H$. We say that $P$ is \textit{boundary incompressible} if there is no boundary compressing disk for $P$ in $E(H)$.\\
The surface $P$ is $\textit{essential}$ if it is incompressible and boundary incompressible.
\end{defn}

In this section, we present handlebody-knots whose exteriors contain meridional incompressible planar surfaces with $n$ boundary components for any $n\geq 2$. This embedding will later be used in the proof of Theorem \ref{main}. We consider next the definition of handlebody-knot.

\begin{defn}\label{handlebody-knot}
A \textit{handlebody-knot} of genus $g$ in $S^3$ is an embedded handlebody of genus $g$ in $S^3$. A \text{spine} $\gamma$ of a handlebody-knot $\Gamma$ is an embedded graph in $S^3$ with $\Gamma$ as a regular neighborhood.
\end{defn}

Let $\Gamma$ be the genus two handlebody-knot $4_1$ from the list of \cite{hknot}, with spine $\gamma$, as in Figure \ref{Hknot41.pdf}. Consider also a collection of distinct knots $C_i$, for $i\in \mathbb{N}$ and $C$ some other non-trivial knot. We work with $\gamma$ as if defined by two vertices, two loops $e_1$, $e_2$, one for each vertex, and an edge $e$ between the two vertices.\\

\begin{figure}[htbp]
\centering
\includegraphics[width=0.3\textwidth]{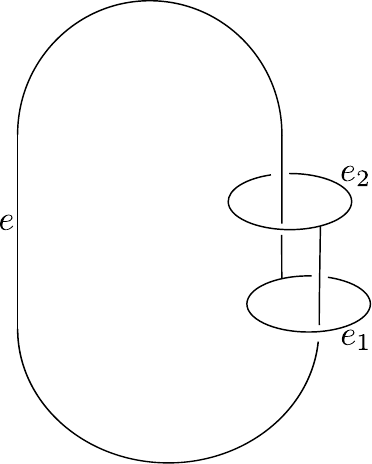}
\caption{: The spine $\gamma$ of the handlebody-knot $\Gamma$, with labels of the two loops $e_1$ and $e_2$, and the of edge $e$.}
\label{Hknot41.pdf}
\end{figure}

Consider two disjoint closed arcs $a_1$ and $a_2$ in $e$, as in Figure \ref{Gammai.pdf}(a). In this figure we also have represented an embedded $2$-sphere $S_2$ in $S^3$ that intersects $\gamma$ in $e$ at two points, $p_1$ and $p_2$, and separates the arcs $a_1$ and $a_2$. Denote the ball bounded by $S_2$ containing a single component of $e$ by $B_{2,1}$ and the other by $B_{22}$. Denote by $l_1$, resp. $l_2$, the component of $B_{22}\cap\gamma$ that contains $e_1$, resp. $e_2$, and note that $l_j$ intersects $S_2$ at $p_j$, $j=1, 2$.\\
We proceed to a connected sum operation between $\gamma$ and the knots $C$ and $C_i$ along the arcs $a_1$ and $a_2$ with an usual connect sum operation. That is, we take a ball in $S^3$ intersecting $\gamma$ in $a_1$, and a ball in $S^3$ intersecting $C_i$ at a single unknotted arc. A connected sum operation is obtained by removing both balls and gluing  their boundaries through a homeomorphism in a way that the boundary points of $a_1$ are mapped to the boundary points of the chosen arc in $C_i$. A similar operation is obtained from the arc $a_2$ and $C$. From these operations we get the handlebody-knots as represented schematically in Figure \ref{Gammai.pdf}(b), that we denote by $\Gamma_i$ with a respective spine $\gamma_i$.
For each handlebody-knot $\Gamma_i$ we consider the swallow-follow torus $X_i$ defined by the connected sum of $C$ with $C_i$. A minimal JSJ- decomposition for the complement of $\Gamma_i$ is defined by the torus $X_i$, cutting from $E(\Gamma_i)$ the exterior of $C_i\# C$, and a JSJ-decomposition of $E(C_i\# C)$. Also, the torus $X_i$ cuts from $E(\Gamma_i)$ the only component obtained from the JSJ-decomposition containing the boundary of $E(\Gamma_i)$. Hence, from the unicity of minimal JSJ-decomposition of compact 3-manifolds, for any other minimal JSJ-decomposition of $E(\Gamma_i)$ the torus cutting the component with the boundary of $E(\Gamma_i)$ is isotopic to $X_i$. Consequently, if $\Gamma_i$ is ambient isotopic to $\Gamma_j$, $i\neq j$, the torus $X_i$ is isotopic to $X_j$, which means $E(C_i\# C)$ is ambient isotopic to $E(C_j\# C)$. This is a contradiction with the torus $C_i\# C$ and $C_j\# C$ being distinct. Then, the handlebody-knots $\Gamma_i$ are not ambient isotopic.\\

\begin{figure}[htbp]
\centering
\includegraphics[width=0.7\textwidth]{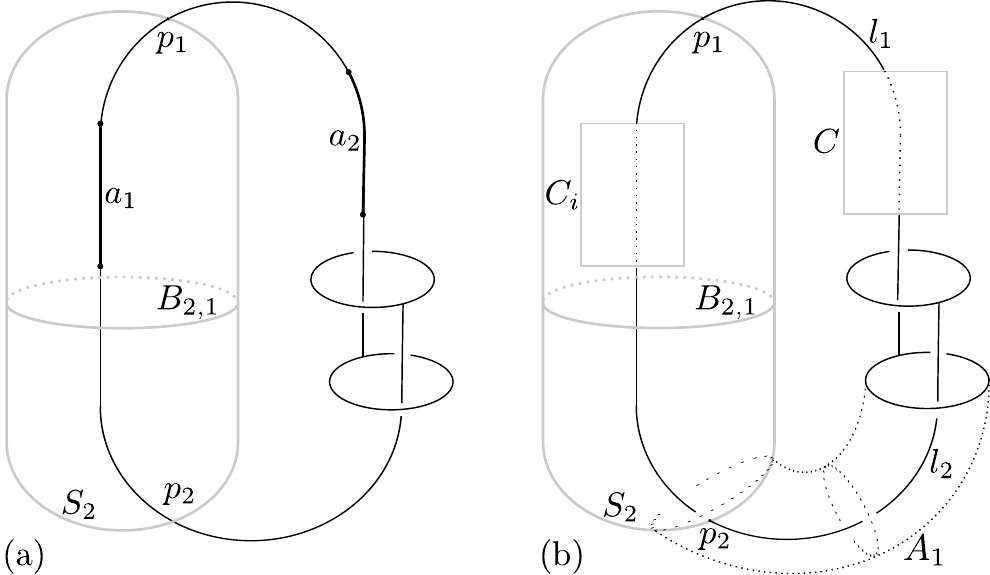}
\caption{: (a) The arcs $a_1$ and $a_2$ in $\gamma$ and the sphere $S_2$; (b) The spines $\gamma_i$ of the handlebody-knots $\Gamma_i$ and the annulus $A_1$. Note that $C_i$ and $C$ label the pattern of the respective knots.}
\label{Gammai.pdf}
\end{figure}

Both loops $e_1$ and $e_2$ co-bound an embedded annulus in $B_{2,2}$, parallel to the component of $e$ in $B_{2,2}$ each encircles, with interior disjoint from $\gamma_i$ and intersecting $S_2$ in the other boundary component. Consider such an annulus with a boundary component in $e_1$, denoted $A_{1}$, as it is illustrated in Figure \ref{Gammai.pdf}(b). We proceed with an isotopy of $\gamma_i$ along $A_{1}$ taking $l_1$ passing through $S_2$ and we obtain $\gamma_i$ as in Figure \ref{GammaiS34.pdf}(a). We refer to this isotopy as an \textit{annulus isotopy} of $\gamma_i$. After this isotopy we denote $S_2$ by $S_3$, considering its relative position with $\Gamma_i$, and the respective balls it bounds by $B_{3,1}$ and $B_{3,2}$. We assume that $l_1$ intersects $S_3$ at $p_1$. Note that all intersections of $\gamma_i$ and $S_3$ are in the arc of $e$ between $p_1$ and $p_2$. Again, we consider an embedded annulus $A_2$ in $B_{3,1}$, co-bounded by $e_1$ and its intersection with $S_3$, parallel to the component of $e\cap B_{3,1}$ disjoint from $e_1$ and in the direction of the local knot $C_i$, following its pattern. By an annulus isotopy of $\gamma_i$  along $A_{2}$ taking $l_1$ passing through $S_3$ we obtain $\gamma_i$ as in Figure \ref{GammaiS34.pdf}(b). After this isotopy we denote $S_3$ by $S_4$, considering its relative position with $\Gamma_i$, and the respective balls it bounds by $B_{4,1}$ and $B_{4,2}$. The ball $B_{4,1}$ intersects $\gamma_i$ in two parallel arcs, and we still assume that $l_1\cap S_4$ is $p_1$. Note again that all intersections of $\gamma_i$ and $S_4$ are in the arc of $e$ between $p_1$ and $p_2$.\\

\begin{figure}[htbp]
\centering
\includegraphics[width=0.7\textwidth]{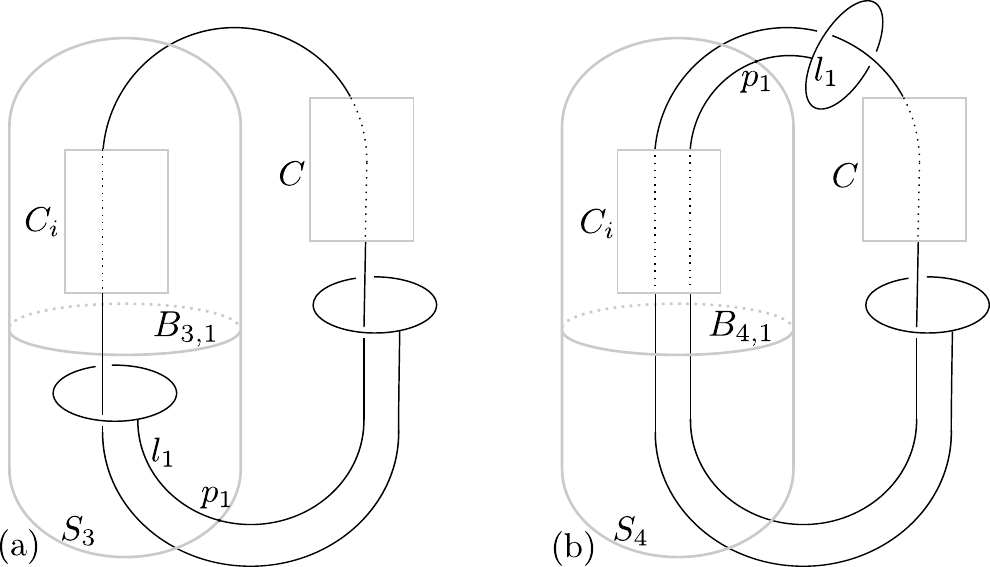}
\caption{: The spine $\gamma_i$ after one, (a), and two, (b), annulus isotopies and the spheres $S_3$ and $S_4$.}
\label{GammaiS34.pdf}
\end{figure}

For a canonical position, we isotope $e_1$ along the component of $e\cap B_{4,2}$, disjoint from $e_1$ and $e_2$, encircling $l_2$. (See Figure \ref{GammaiS45.pdf}(a).) We can now continue the previous process. Consider again an annulus $A_3$ in $B_{4,2}$, co-bounded by $e_1$ and its intersection with $S_4$, parallel to the components of $e\cap B_{4,2}$ other than $l_1$, and in the opposite direction of the local knot $C$. By an annulus isotopy of $\gamma_i$ along $A_3$ taking $l_1$ passing through $S_4$ we obtain $\gamma_i$ as in Figure \ref{GammaiS45.pdf}(b). After this isotopy we denote $S_4$ by $S_5$, considering its relative position with $\Gamma_i$, and the respective balls it bounds by $B_{5,1}$ and $B_{5,2}$. Again, $l_1$ intersects $S_5$ at $p_1$, and all intersections of $S_5$ with $\gamma_i$ are in the arc of $e$ between $p_1$ and $p_2$. For the next step proceed with an annulus isotopy along an annulus $A_4$ in $B_{5,1}$ co-bounded by $e_1$, parallel to the components of $e\cap B_{5,1}$ disjoint from $e_1$, in the direction of the local knot $C_i$, following its pattern.\\

\begin{figure}[htbp]
\centering
\includegraphics[width=0.7\textwidth]{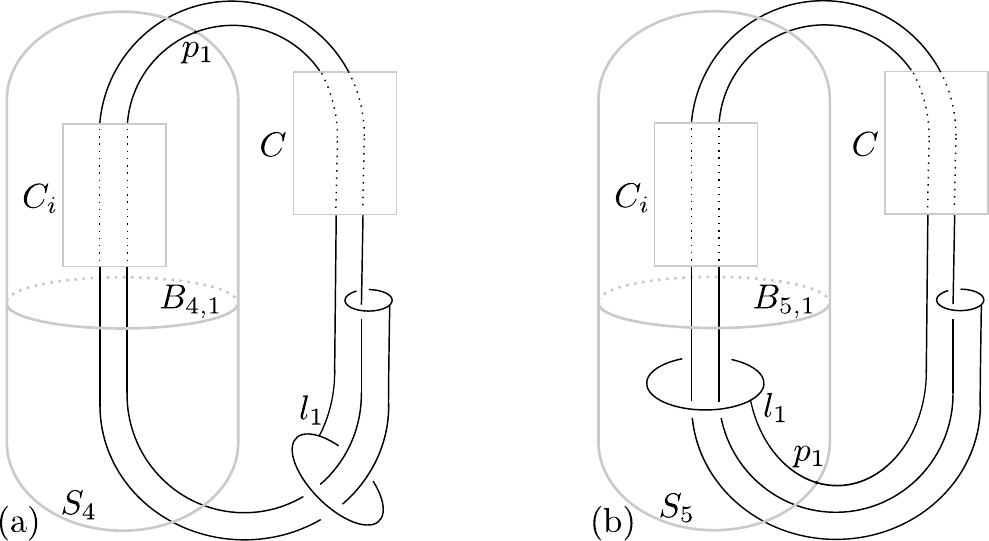}
\caption{: In (a) we have the spine $\gamma_i$ of Figure \ref{GammaiS34.pdf}(b) in a canonical position, and in (b) we have $\gamma_i$ after another annulus isotopy.}
\label{GammaiS45.pdf}
\end{figure}

After $2(k-1)$, $k=1, 2, \ldots$, annulus isotopies as the ones explained above we get $\gamma_i$ as in Figure \ref{GammaiSn.pdf}(a). From $S_2$ we obtain $S_{2k}$, and the respective balls it bounds, $B_{2k,1}$ and $B_{2k,2}$. The ball $B_{2k,1}$ intersects $\gamma_i$ in $k$ parallel arcs with the pattern of $C_i$, and the ball $B_{2k,2}$ intersects $\gamma_i$ in $k-2$ parallel arcs with the pattern of $C$, another arc with the pattern of $C$ encircled by $l_2$, and $l_1$ that encircles all these other components.\\
After $2k-1$, $k=1, 2, \dots$, annulus isotopies we obtain $\gamma_i$ as in Figure \ref{GammaiSn.pdf}(b). From $S_2$ we obtain $S_{2k+1}$, and the respective balls it bounds, $B_{2k+1,1}$ and $B_{2k+1,2}$. The ball $B_{2k+1,1}$ intersects $\gamma_i$ in $n$ parallel arcs with the pattern of $C_i$ and $l_1$ encircling these arcs, and the ball $B_{2k+1,2}$ intersects $\gamma_i$ in $k-1$ parallel arcs with the pattern of $C$, together with another arc with the pattern of $C$ and $l_2$ which encircles this arc.\\
Note after each isotopy we assume that $l_j$ intersects $S_n$, $n=2, 3, \ldots$, in $p_j$ and that all points of $S_n\cap \gamma_i$ are in the arc between $p_1$ and $p_2$ in $e$.\\

\begin{figure}[htbp]
\centering
\includegraphics[width=0.7\textwidth]{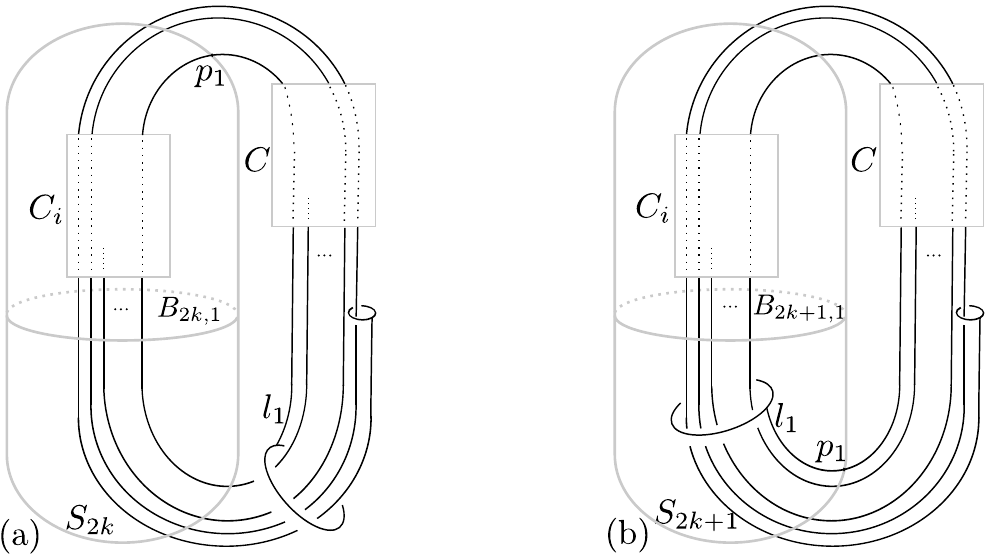}
\caption{: The spine $\gamma_i$ after an even number, in (a), and an odd number, in (b), of annulus isotopies and the corresponding spheres $S_{2k}$ and $S_{2k+1}$, $k\in \mathbb{N}$.}
\label{GammaiSn.pdf}
\end{figure}

We denote $S^3-int\, \Gamma_i$ by $E(\Gamma_i)$, and $S^3-\gamma_i$ by $E(\gamma_i)$. Let $Q_n$, for $n=2, 3, \ldots$, be the intersection of $S_n$ with $E(\Gamma_i)$ in $S^3$.

\begin{lem}\label{Qn}
The surface $Q_n$ is incompressible in $E(\Gamma_i)$.
\end{lem}
\begin{proof}
As $\Gamma_i$ is a regular neighborhood of $\gamma_i$, if $Q_n$ is compressible in $E(\Gamma_i)$ then $S_n$ is compressible in $E(\gamma_i)$. Hence, it suffices to prove that $S_n$ is incompressible in $E(\gamma_i)$.\\
1. Suppose $n$ is even. Then $S_n$ is as in Figure \ref{GammaiSn.pdf}(a).\\
(i) In this case, the ball $B_{n,1}$ intersects $\gamma_i$ in a collection of $k=\frac{n}{2}$ parallel knotted arcs. Then $(B_{n,1}, B_{n,1}\cap \gamma_i)$ is an essential tangle. In fact, suppose there is a compressing disk $D$ for $S_n$ in $B_{n,1}-(B_{n,1}\cap \gamma_i)$. Then $D$ separates the arcs $B_{n,1}\cap \gamma_i$ into two collections. Let $s_1$ and $s_2$ be two arcs in $B_{n,1}$ which are separated by $D$. As $s_1$ and $s_2$ are parallel there is a disk $E$ with boundary $s_1\cup s_2$ and two arcs in $S_n$, $\alpha_1$ and $\alpha_2$, each with one end in $s_1$ and the other in $s_2$. Consider $D$ and $E$ in general position and suppose that $|D\cap E|$ is minimal. If $D$ intersects $E$ in simple closed curves or in arcs with both ends in $\alpha_1$ or both in $\alpha_2$, by an innermost arc type of argument we can reduce $|D\cap E|$, which is a contradiction. Therefore, all arcs of $D\cap E$ have one end in $\alpha_1$ and the other end in $\alpha_2$. Hence, both $s_1$ and $s_2$ are parallel to outermost arcs of $D\cap E$ in $D$, which implies that $s_1$ and $s_2$ are parallel to $S_n$. This is a contradiction because the arcs $s_1$ and $s_2$ are knotted by construction.\\
(ii) If $n\leq 4$ then the ball $B_{n,2}$ intersects $\gamma_i$ in $l_1$, $l_2$, and when $n=4$ also in an arc encircled by both $l_1$ and $l_2$. In this case if there is a compressing disk for $S_n$ in $B_{n,2}-(B_{n,2}\cap \gamma_i)$ it separates the components $l_1$ or $l_2$ from the other components. This implies that $e_1$ or $e_2$ bound a disk in the complement of $\gamma_i$, which is a contradiction with $\Gamma_i$ being a knotted handlebody-knot. Otherwise, suppose that $n>4$. Thus, $B_{n,2}$ intersects $\gamma_i$  in $\frac{n}{2}-2$ parallel arcs with the pattern of $C$, another arc with the pattern of $C$ encircled by $l_2$, and the component $l_1$ that encircles the arc encircled by $l_1$ and the $\frac{n}{2}-2$ parallel arcs. With exception to $l_1$ and $l_2$, all other arcs are parallel as properly embedded arcs in $B_{n, 2}$. Thus, if a compressing disk for $S_n$ in $B_{n,2}-(B_{n,2}\cap \gamma_i)$ separates these arcs, following an argument as in 1(i) we have a contradiction with these arcs being knotted. Therefore, a compressing disk for $S_n$ in $B_{n,2}-(B_{n,2}\cap \gamma_i)$ separates a single component $l_1$ or $l_2$ from all the other components, or it separates both components $l_1$ and $l_2$ from the other parallel arcs. As $e_1$ bounds a disk disjoint from $l_2$, in both cases $e_1$ bounds a disk in the complement of $\gamma_i$, which is a contradiction with $\Gamma_i$ being a knotted handlebody-knot.\\

2. Suppose now that $n$ is odd. Then $S_n$ is as in Figure \ref{GammaiSn.pdf}(b).\\
(i) The ball $B_{n,1}$ intersects $\gamma_i$ in a collection of $\frac{n-1}{2}$ parallel arcs and $l_1$ which encircles these arcs. If there is a compressing disk $D$ of $S_n$ in $B_{n,1}-B_{n,1}\cap \gamma_i$ separating the parallel arcs, following an argument as in 1(i) we have a contradiction with these arcs being knotted. If $D$ separates the component $l_1$ from the other components, following an argument as in 1(ii) we have a contradiction with $\Gamma_i$ being a knotted handlebody-knot.\\
(ii) If $n=3$ the ball $B_{n,2}$ intersects $\gamma_i$ in an arc with pattern $C$ and $l_2$ which encircles the arc. If there is a compressing disk for $S_n$ in $B_{n,2}-(B_{n,2}\cap\gamma_i)$ in this case, then it separates the component $l_2$ from the arc with pattern $C$. From the same argument used in 1(ii) we have a contradiction with $\Gamma_i$ being a knotted handlebody-knot. If $n>3$ then the ball $B_{n,2}$ intersects $\gamma_i$ in $\frac{n-1}{2}$ parallel arcs, and $l_2$ which encircles one of the previous arcs. Without considering $l_2$, if a compressing disk for $S_n$ in $B_{n,2}-(B_{n,2}\cap \gamma_i)$ separates the parallel arcs then following an argument as in 1(i) we have a contradiction with the arcs being knotted. Then, if $S_n$ has a compressing disk in $B_{n,2}-(B_{n,2}\cap \gamma_i)$ then this disk isolates the component $l_2$ from the other components, and following the argument as in 1(ii) we have a contradiction with $\Gamma_i$ being a knotted handlebody-knot.
\end{proof}

The surface $Q_n$ is boundary compressible in $E(\Gamma_i)$, as there are boundary compressing disks over the regular neighborhoods of $l_1$ and $l_2$. However, our construction of the handlebody-knots $\Gamma_i$ could have been made in a way that the surfaces $Q_n$ are incompressible and boundary incompressible in their complements. For that purpose, we could do a connect sum of $\gamma_i$ with two knots along two arcs in $e_1$ and $e_2$. After this operation, there won't be boundary compressing disks of $Q_n$ over the regular neighborhoods of $l_1$ and $l_2$ in $E(\Gamma_i)$. And as these are the only possible boundary compressing disks, because all other components $\gamma_i-\gamma_i\cap S_n$ correspond to knotted arcs in their respective balls, after these connected sums the surfaces $Q_n$ would also be boundary incompressible in the complement of the handlebody-knots. But for the purpose of this paper, we will use the handlebody-knots $\Gamma_i$.

\section{Knots with essential tangle decompositions with arbitrarily high number of strings}\label{examples}

In this section we use the handlebody-knots $\Gamma_i$ to construct infinitely many examples of knots with essential tangle decompositions for all numbers of strings.\\
Let $N_1$ and $N_2$ be torus knots in the boundary of the solid tori $T_1$ and $T_2$ (that we assume to be in different copies of $S^3$). Consider $B_i$ a regular neighborhood of an arc of $N_i$ intersecting $T_i$ at a ball, for $i=1, 2$. We isotope $B_i$ and $B_i\cap N_i$ away from the interior of $T_i$ such that $B_i$ intersects $T_i$ at a disk, for $i=1, 2$. We proceed with a connect sum of $N_1$ and $N_2$ by removing the interior of $B_1$ and attaching the exterior of $B_2$ in a way that the disks $B_1\cap T_1$ and $B_2\cap T_2$ are identified. Hence, the knot $N_1\# N_2$, denoted by $K$, is in the boundary of a genus two handlebody $H$, obtained by gluing $T_1$ and $T_2$ along a disk in their boundaries. We denote the identification disk of $T_1$ and $T_2$ in $H$ by $D$. In Figure \ref{K.pdf} we have the example of this connected sum with two trefoils, that we will use as reference for the remainder of the paper.

\begin{figure}[htbp]
\centering
\includegraphics[width=0.5\textwidth]{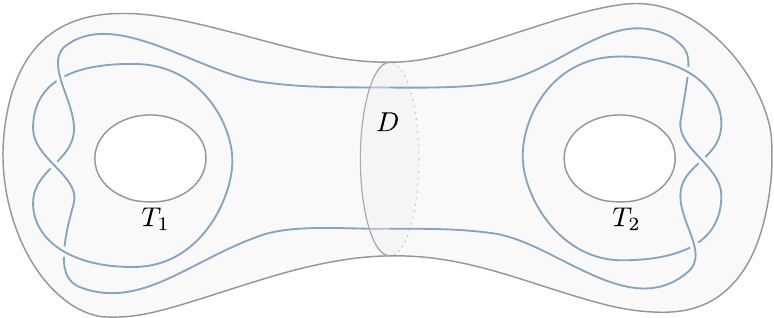}
\caption{: The handlebody $H$ with the connected sum of two trefoil knots.}
\label{K.pdf}
\end{figure}

\noindent Consider disks $D_1$ and $D_2$ parallel to $D$ in $H$, such that the cylinder $C_{1,2}$ cut by $D_1\cup D_2$ from $H$ intersects $K$ in two parallel arcs, each with one end in $D_1$ and the other in $D_2$. We also keep denoting by $T_1$ and $T_2$ the solid torus cut from $H$ by $D_1$ and $D_2$, respectively.  (See Figure \ref{Kspine.pdf}.) Let $s$ be a spine of $H$ that intersects $C_{1,2}$ in a single arc. We denote by $d_i$ the point $D_i\cap s$, and by $t_i$ the intersection of $s$ with $T_i$, for $i=1, 2$.

\begin{figure}[htbp]
\centering
\includegraphics[width=0.5\textwidth]{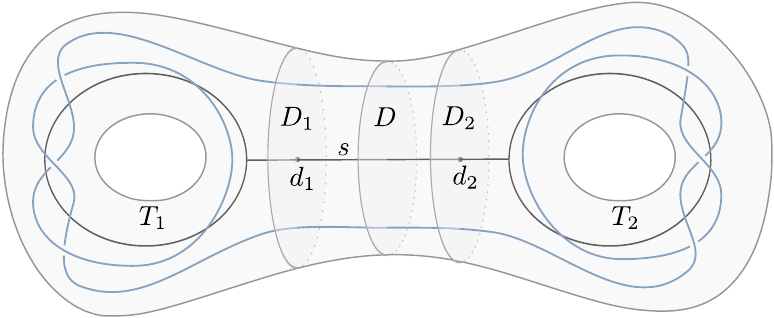}
\caption{: The handlebody $H$ and the spine $s$ with the connected sum of two trefoil knots.}
\label{Kspine.pdf}
\end{figure}

\noindent We now embed the knot $K$ in $\Gamma_i$ as follows. Consider an embedding $h_i$ of $H$ in $S^3$ taking $H$ homeomorphically to $\Gamma_i$, such that $h_i(s)=\gamma_i$, $h_i(d_j)=p_j$, $h_i(t_j)=l_j$ and also that $h_i(T_j)=L_j$, for $j=1, 2$.

\begin{proof}[Proof of Theorem \ref{main}]
Denote by $K_i$ the knots $h_i(K)$, $i\in \mathbb{N}$, for a fixed knot $K$. As to prove the handlebody-knots $\Gamma_i$ being distinct, let $X_i$ be the torus cutting from $E(K_i)$ the exterior of $C_i\# C$. The component cut by $X_i$ from $E(K_i)$ containing the boundary torus is the same for every knot $K_i$. Hence, from the unicity of minimal JSJ-decomposition of compact 3-manifolds, if two knots $K_i$ and $K_j$ are ambient isotopic the tori $X_i$ and $X_j$ are also ambient isotopic, contradicting $C_i\#C$ and $C_j\#C$ being distinct. Then, the knots $K_i$ define a collection of distinct knots.\\
To prove the statement of the theorem, we will show that the spheres $S_n$, $n\geq 2$, define  $n$-string essential tangle decomposition for the knots $K_i$, and that these knots are prime.\\

We start by proving that $S_n$ defines an $n$-string essential tangle decomposition of $K_i$. Let $E(K_i)$ be the exterior of $K_i$ in $S^3$, that is $S^3-int\, N(K_i)$, and let $P_n$ be the intersection of $S_n$ with $E(K_i)$, for a fixed $n$. To prove that $S_n$ defines an essential tangle decomposition for $K_i$, we need to prove that $P_n$ is essential in $E(K_i)$, \textit{i.e.} that $P_n$ is incompressible and boundary incompressible.\\
First, we observe that $P_n$ is boundary incompressible. In fact, as the strings of $K\cap B_{n,i}$ in $B_{n,i}$, $i=1, 2$, are knotted, there is no boundary compressing disk for $P_n$ in $E(K_i)$.\\
Now we prove that $P_n$ is incompressible in $E(K_i)$. Let $\Delta_j$, $j=1, \ldots, n$, be the disks of intersection between $\Gamma_i$ and $S_n$ with $\Delta_1=L_1\cap S_n$ and $\Delta_n=L_2\cap S_n$. Denote by $C_{j, j+1}$ the cylinder cut by $\Delta_j\cup \Delta_{j+1}$ from $\Gamma_i$. Denote also by $\partial^* C_{j, j+1}$ the annulus $C_{j, j+1}\cap \partial \Gamma_i$, that is $\partial C_{j, j+1}-(\Delta_j\cup \Delta_{j+1})$. Note that $C_{j, j+1}\cap K$ is a collection of two arcs parallel to $\partial^* C_{j, j+1}$, each with one end in $\Delta_j$ and the other in $\Delta_{j+1}$. Let us consider also $\partial^* L_1$ and $\partial^* L_2$ to denote $\partial L_1 - \Delta_1$ and $\partial L_2 - \Delta_n$. Furthermore, we denote by $s_j$ the string component, of the tangle decomposition of $K_i$ defined by $S_n$, in $L_j$, $j=1, 2$. Note that $s_j$ is parallel to $\partial^* L_j$. We isotope $s_j$ into $\partial^* L_j$ and denote the annulus $\partial^* L_j\cap E(K_i)$ by $\Lambda_j$.\\
Suppose that $P_n$ is compressible in $E(K_i)$ with $D$ a compressing disk, properly embedded in $B_{n,1}$ or $B_{n,2}$, in general position with $\Gamma_i$.  If $D$ is disjoint from $\Gamma_i$ we have a contradiction with Lemma \ref{Qn}. In this way, we assume that $D$ intersects $\Gamma_i$ and that $|D\cap \partial \Gamma_i|$ is minimal over all isotopy classes of compressing disks of $P_n$ in $E(K_i)$.\\ 
In particular, assume that $D$ intersects an annulus $\partial^* C_{j, j+1}$. If $D\cap \cup_{j=1}^{n-1}\partial^* C_{j, j+1}$ contains a simple closed curve or an arc with both ends in the same disk of $\Gamma_i\cap S_n$, by considering an outermost one between such curves and arcs in $\partial^* C_{j, j+1}$, and by cutting and pasting along the disk it bounds or co-bounds, we get a contradiction with the minimality of $|D\cap \partial \Gamma_i|$. Thus, $D\cap \cup_{j=1}^{n-1}\partial^* C_{j, j+1}$ is a collection of arcs with ends in distinct disks of $\Gamma_i\cap S_n$. Consider an outermost arc of $D\cap \cup_{j=1}^{n-1}\partial^* C_{j, j+1}$ in $D$, say $a$, and, without loss of generality, suppose it belongs to $\partial^* C_{j, j+1}$. The arc $a$ is parallel to a string of the tangle defined by $S_n$ that is in $C_{j, j+1}$, which contradicts the fact that all strings of the tangle decomposition of $K_i$ defined by $S_n$ are knotted. Consequently, we can assume that $D\cap \cup_{j=1}^{n-1}\partial^* C_{j, j+1}$ is empty.\\
Then, we are assuming that $D$ intersects $\partial \Gamma_i$ at $\partial^* L_1$ or $\partial^* L_2$, or more precisely at $\Lambda_1$ or $\Lambda_2$. We denote by $a_j$ and $a_j'$ the arcs of $\partial \Lambda_j$ parallel to $s_j$ in $\partial^* L_j$, and by $b_j$ and $b_j'$ the arcs cut by $\partial a_j$ and $\partial a_j'$, respectively, in the boundary of $\partial^* L_j$. The boundary components of $\Lambda_j$ are $a_j\cup b_j$ and $a_j'\cup b_j'$. Note that, as $D\cap s_j$ is empty, the disk $D$ is disjoint from $a_j$ and $a_j'$. Note also that $a_j\cup b_j$ is a torus knot in the torus $\partial^* L_j\cup (S_n-L_j\cap S_n)$, denoted $T_j'$. If $D$ intersects $\Lambda_j$ in innessential simple closed curves or arcs with both ends in $b_j$ or both ends in $b_j'$ then, by cutting and pasting along a disk cut by such curve or arc, we have a contradiction with the minimality of $|D\cap \partial \Gamma_i|$. If $D$ intersects $\Lambda_j$ in an essential simple closed curve then $a_j\cup b_j$ is parallel to a simple closed curve in $D$, which contradicts $a_j\cup b_j$ being knotted. Consequently, $D$ intersects $\Lambda_j$ in a collection of arcs each with one end in $b_j$ and the other in $b_j'$. Let $O$ be an outermost disk in $D$ cut by the arcs of $D\cap \Lambda_j$. Then, $O$ is a disk in a solid torus bounded by $T_j'$ and intersects the torus knot $a_j\cup b_j$ in $T_j'$ at a single point. As we are working in $S^3$, either $O$ is parallel to $T_j'$ or it is a meridian to a solid torus bounded by $T_j$. In either way, $O$ intersects any torus knot in $T_j'$ at least in two points, which contradicts $O$ intersecting $a_j\cup b_j$ once.\\
Therefore, we have that $P_n$ is essential in the complement of $K_i$, which ends the proof that $S_n$ defines an $n$-string essential tangle decomposition of $K_i$.\\


Now we prove that the knots $K_i$ are prime. From Theorem 1 of \cite{Bleiler}, if a knot has a $2$-string prime tangle decomposition, that is the tangles are essential and with no local knots, the knot is prime. We have that the knot $K_i$ has a $2$-string essential tangle decomposition defined by $S_2$. So, to prove that it is prime, we just need to show that the tangle decomposition defined by $S_2$ has no local knots. The ball $B_{2,1}$ intersects $K_i$ in two parallel arcs. Hence, if there is a $2$-sphere intersecting only one of the arcs at a single component, this component has to be unknotted. The ball $B_{2,2}$ intersects $\gamma_i$ in $l_1$ and $l_2$, then it intersects $K_i$ at two strings each with the pattern of a torus knot. Note that, even though the pattern of the knot $C$ is in $l_2$, it does not affect the topological type of the string in $L_2$. Suppose, the tangle in $B_{2,2}$ contains a local knot. That is, there is a ball $Q$ intersecting only one of the strings, and at a knotted arc. As the torus knots are prime, this knotted arc contains the all pattern of the string, that is the intersection of $Q$ and $B_{2,2}$ with this string is topologically the same. Therefore, as the strings in $B_{2,2}$ are parallel to the boundary of $L_1$ and $L_2$ and $Q$ intersects only one of them, we have that $Q$ contains either $e_1$ or $e_2$ or we can isotope $e_1$ and $e_2$ in a way that $Q$ contains either $e_1$ or $e_2$. But then, either $e_1$ or $e_2$ bound a disk in the complement of $\gamma_i$ and, as in 1(ii) from the proof of Lemma \ref{Qn}, we have a contradiction with $\Gamma_i$ being a knotted handlebody-knot. Consequently, the tangle decomposition defined by $S_2$ contains no local knots and the knots $K_i$ are prime.
\end{proof}

Corollary \ref{cor} is now an immediate consequence.

\begin{proof}[Proof of Corollary \ref{cor}]
In Theorem \ref{main} we proved that the spheres $S_n$, $n\geq 2$, define a $n$-string essential tangle decomposition for the knots $K_i$. Hence, considering the knots $K_i$ connected sum with some other knot, we have infinitely many knots with $n$-string essential tangle decompositions for all $n\in \mathbb{N}$, as in the statement of this corollary.
\end{proof}

\section*{Acknowledgements}
The author thanks to Cameron Gordon for comments along the development of this paper and the refferee for the detailed review and suggestions on the exposition and proofs of the paper.

\end{document}